\documentclass[a4paper]{amsart}
\usepackage{graphicx}
\theoremstyle{plain}
  \newtheorem{thm}{Theorem}[section]
  \newtheorem{lem}[thm]{Lemma}
  
  \newtheorem{prop}[thm]{Proposition}

  \newtheorem{stmt}[thm]{Statement}
\theoremstyle{definition}
  
  \newtheorem{ex}[thm]{Example}

\theoremstyle{remark}
  \newtheorem{rem}[thm]{Remark}
  \newtheorem*{ack}{Acknowledgments}
\newcommand{\Z}{\mathbb{Z}}

\newcommand{\Znonneg}{\mathbb{Z}_{\geq0}}

\newcommand{\R}{\mathbb{R}}
\newcommand{\basedMatrices}[1]{\mathcal{B}_#1}
\newcommand{\skewSymMatrices}[1]{\mathcal{S}_#1}

\newcommand{\cover}[2]{#1^{(#2)}}

\newcommand{\nanoword}[2]{#1\colon\hspace{-1mm}#2}
\newcommand{\link}[2]{l(#1,#2)}
\newcommand{\trivial}{0}
\newcommand{\xType}[1]{\lvert#1\rvert}

\newcommand{\textcover}[1]{$#1$-covering}
\newcommand{\factorial}[1]{#1!}

\newcommand{\vstring}[2]{$#1_{#2}$}
\newcommand{\surfdiag}{virtual string surface diagram}

\DeclareMathOperator{\n}{n}
\DeclareMathOperator{\swap}{swap}
\DeclareMathOperator{\flatten}{flat}
\numberwithin{equation}{section}
\allowdisplaybreaks
\begin{document}
\title{On Tabulating Virtual Strings}
\author{Andrew Gibson}
\address{
Department of Mathematics,
Tokyo Institute of Technology,
Oh-okayama, Meguro, Tokyo 152-8551, Japan
}
\email{gibson@math.titech.ac.jp}
\date{\today}
\begin{abstract}
A virtual string can be defined as a closed curve on a surface modulo
 certain equivalence relations. 
Turaev defined several invariants of virtual strings which we use to
 produce a table of virtual strings up to 4 crossings.
We discuss progress in extending the tabulation to 5 crossings.
We also provide a counter-example to a statement of Kadokami.
\end{abstract}
\thanks{This work was supported by a Scholarship from the Ministry of
Education, Culture, Sports, Science and Techonology of Japan}
\maketitle
\section{Introduction}
In \cite{Kauffman:VirtualKnotTheory} Kauffman introduced the idea of a
virtual knot by use of diagrams on the plane. A virtual knot diagram is 
an immersion of a circle in $\R^2$ with a finite
number of self-intersections of the circle. Each self-intersection is a
transverse double point which we call a crossing.
There are two types of crossings. A real crossing is a crossing
where one arc is specified to be over-crossing and the other
under-crossing. A virtual crossing is a crossing where no over or under
crossing information is specified.
Each virtual crossing in a diagram is marked with a small circle.
Figure~\ref{fig:kishino} shows an example of such a diagram.
Virtual knots can then be
defined to be equivalence classes of these diagrams under a relation
given by a set of diagrammatic moves based on the Reidemeister
moves.
\par
\begin{figure}[hbt]
\begin{center}
\includegraphics{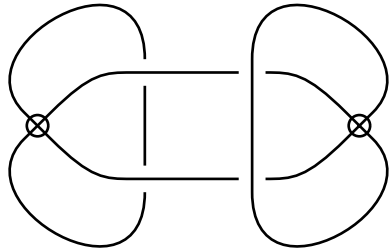}
\caption{A virtual knot diagram of a virtual knot known as Kishino's knot}
\label{fig:kishino}
\end{center}
\end{figure}
Carter, Kamada and Saito showed that we can consider virtual knots as
equivalence classes of knot diagrams on compact oriented surfaces
\cite{CKS:StableEquivalence}. Here the knot diagrams only use real
crossings. Knot diagrams are considered equivalent if they are related
by a finite sequence of Reidemeister moves or stable
homeomorphisms (which allow parts of the surface away from the knot to
be changed). 
\par
If we take a knot diagram on a surface and flatten the crossings to
double points we get a curve on the surface. We can flatten the
Reidemeister moves in the same way. We can then define an equivalence
relation on the flattened diagrams using these flattened Reidemeister
moves and stable homeomorphisms. We define a virtual string to be
an equivalence class under this relation. Definitions of the moves
and of stable homeomorphism will be given in the next section.
Elsewhere in the literature virtual strings are also known as projected
virtual knots (for example \cite{Kadokami:Non-triviality}), flat
knots or flat virtual knots (for example
\cite{Hrencecin/Kauffman:Filamentations}),
or universes of virtual knots (for example
\cite{CKS:StableEquivalence}).
\par
It turns out that this `flattening' operation on knot diagrams on
surfaces induces a well-defined map from the set of virtual knots to the
set of virtual strings. In other words, if $D_1$ and $D_2$ are
knot diagrams on surfaces in the same virtual knot equivalence class,
the flattened versions $\flatten(D_1)$ and $\flatten(D_2)$ are in the
same virtual string 
equivalence class. Thus the virtual string underlying a virtual knot is
an invariant of that knot and invariants of virtual strings may be used
to distinguish virtual knots.
\par
For a classical knot $K$ it is always the case that the underlying
virtual string is trivial. On the other hand, consider the knot in
Figure~\ref{fig:kishino}, which is known as Kishino's knot. It is known
that 
this virtual knot is indistinguishable from the trivial virtual knot
using certain virtual knot invariants like the Jones polynomial and the
fundamental group. Kishino showed that it was non-trivial by calculating
the Jones polynomial of the $3$-parallel of the knot
\cite{Kishino/Shin:VirtualKnots}.
Another way to show its non-triviality is to show that its underlying
virtual string is non-trivial. Kadokami suggested this approach in
\cite{Kadokami:Non-triviality}. Kadokami's proof of non-triviality of
the underlying virtual string is based on another theorem in that paper
for which we found a problem with the proof (we discuss this problem in
Section~\ref{sec:kadokami}). Non-triviality of the virtual string can be
shown using Turaev's primitive based matrix invariant. Turaev gave the
calculation of this invariant for this virtual string in
\cite{Turaev:2004} although the connection with Kishino's knot is not
noted there.
\par
The rest of the paper is organised as follows.
\par
In the next section we give a full definition of virtual strings.
In Section~\ref{sec:nanowords} we explain another representation of
virtual strings using Turaev's nanowords. This representation is useful
for giving tables of virtual strings.
In Section~\ref{sec:invariants} we briefly explain some invariants of
virtual strings defined by Turaev in \cite{Turaev:2004}.
\par
In Sections~\ref{sec:canonical} and~\ref{sec:enumeration} we give some
details of the algorithms we used for canonicalizing the representation
of primitive based matrices and for the enumeration of virtual
strings. 
The main results of this paper appear in Section~\ref{sec:results}
where the results of our enumeration are given. 
In the final section we discuss a statement in a paper by
Kadokami \cite{Kadokami:Non-triviality} and give a counter-example to
it. 
\par
The contents of this paper form part of the author's thesis submitted to
fulfil a requirement of the Master's course at Tokyo Institute of
Techonology.
\begin{ack}
The author would like to thank the organizers of the International
 Conference on Quantum Topology 6-12th August 2007 for organizing an
 interesting conference. He is also grateful to the members of the Institute
 of Mathematics, Hanoi for their hospitality during the conference. The
 author also thanks his supervisor Hitoshi Murakami and Yuji Terashima
 for their help.
\end{ack}
\section{Virtual strings}
A \emph{\surfdiag} is a pair $(S,D)$, where $S$ is
a compact oriented surface and $D$ is an immersion of an oriented circle
in $S$. 
Self-intersections of the immersion should be transverse double points.
We call these self-intersections \emph{crossings}.
In this paper, we only consider the case where the number of
crossings is finite. Figure~\ref{fig:diagram31} shows an example of such
a pair. Here we consider the outer region to be a disc, and so the surface
$S$ is compact and has genus $2$.
\begin{figure}[hbt]
\begin{center}
\includegraphics{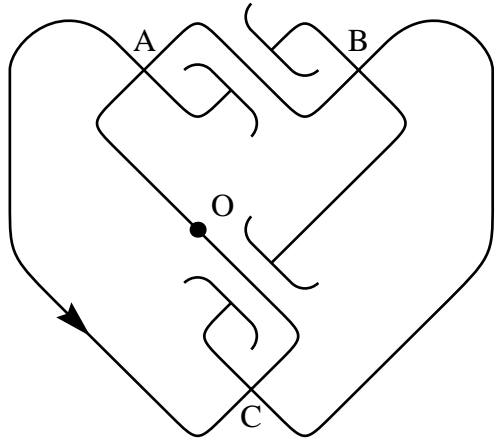}
\caption{A non-trivial virtual string with orientation marked by an
 arrow. The point marked O and the crossing labels A, B and C will be
 used in Section~\ref{sec:nanowords}.}
\label{fig:diagram31}
\end{center}
\end{figure}
\par
We say that two \surfdiag s $(S,D)$ and $(S^\prime, D^\prime)$ are
\emph{stably homeomorphic} if there is a homeomorphism mapping a regular
neighbourhood of $D$ in $S$ to a regular neighbourhood of $D^\prime$ in
$S^\prime$ preserving the orientations of the circle and the
surface. Note that if we have two immersions $D$ and $D^\prime$ on the
same surface $S$ such that $D$ is isotopic to $D^\prime$ as a graph,
$(S,D)$ and $(S, D^\prime)$ are stably homeomorphic. In this way we can
consider isotopy of immersions as a special case of stable
homeomorphism.
\par
\begin{figure}[hbt]
\begin{center}
\includegraphics{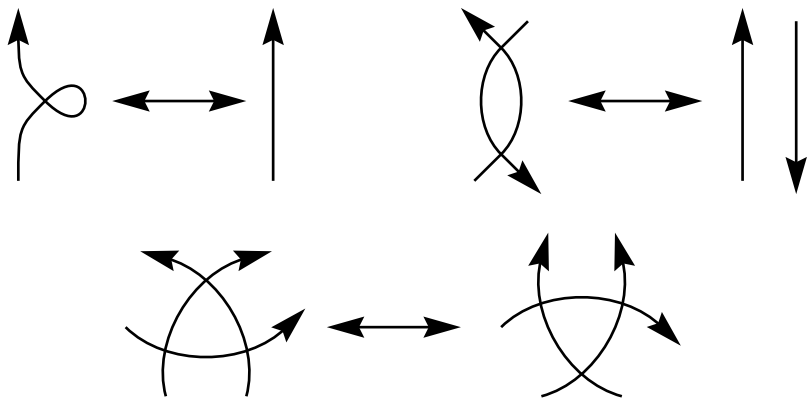}
\caption{The flattened Reidemeister moves}
\label{fig:reidermeister}
\end{center}
\end{figure}
Figure~\ref{fig:reidermeister} gives some moves between two \surfdiag s
$(S,D)$ and $(S, D^\prime)$. Each side of each move shows a small area
of $S$ homeomorphic to a disc. These moves are called flattened
Reidemeister moves because they are derived from the usual Reidemeister
moves of knot theory (see, for example, \cite{Lickorish:1997} for
definitions) by flattening each crossing to a double point. We sometimes
call the flattened Reidemeister moves \emph{homotopy moves}.
\par
We say two \surfdiag s $(S,D)$ and $(S^\prime, D^\prime)$ are
\emph{homotopic} if there exists a finite sequence of stable homeomorphisms
and flattened Reidemeister moves transforming one pair to the
other. Clearly this relation is an equivalence relation and we call it
\emph{homotopy}.
This equivalence relation is also known as \emph{stable equivalence}
\cite{Kadokami:Non-triviality}.
We define a \emph{virtual string} to be an equivalence class of this
relation.
\par
For any move shown in Figure~\ref{fig:reidermeister} we can consider
variants of the move by swapping the orientation of one or more arcs
(and swapping the orientations of the corresponding arcs on the other
side of the move). It is well known that we can derive
all such variants from the moves shown in
Figure~\ref{fig:reidermeister}. We showed how we can do this in 
\cite{Gibson:mthesis}.
This gives us a larger set of moves. Any move in this set which adds or
removes a single crossing is called a $1$-move. Any move in this set
which adds or removes two crossings is called a $2$-move. All remaining
moves in the set involve three crossings and we call them $3$-moves.
\par
Virtual strings can be represented as planar diagrams which we call
\emph{virtual string diagrams}.
Given a \surfdiag{}
$(S,D)$ we can project $D$ to a plane in such a way that any
self-intersections of the image of $D$ are transverse double points and
there are a finite number of them. Any crossings in the image that
correspond to crossings on $S$ are called real crossings. Any other
crossings in the image are called virtual crossings. We mark a virtual
crossing with a small circle (see Figure~\ref{fig:virtualcrossings}). We
can convert a planar diagram back to a \surfdiag{} by replacing any
virtual crossings with handles. Figure~\ref{fig:handle} shows this local
transformation.
\begin{figure}[hbt]
\begin{center}
\includegraphics{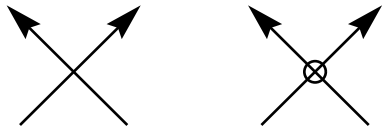}
\caption{A real crossing (left) and a virtual crossing (right)}
\label{fig:virtualcrossings}
\end{center}
\end{figure}
\begin{figure}[hbt]
\begin{center}
\includegraphics{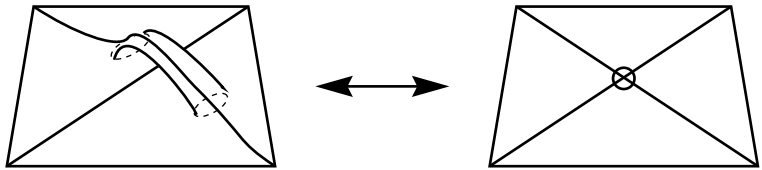}
\caption{Changing surface with a hollow handle (left) to a
planar diagram with a virtual crossing (right)}
\label{fig:handle}
\end{center}
\end{figure}
\par
We can think of a virtual string diagram as a virtual knot diagram
\cite{Kauffman:VirtualKnotTheory} where the real crossings have been
`flattened' to double points, removing the over and under crossing 
information. In the same way we can flatten the generalized Reidemeister
moves of virtual knot theory \cite{Kauffman:VirtualKnotTheory} to get
moves for virtual string diagrams (see, for example,
\cite{Kadokami:Non-triviality}).
We can define a relation on virtual string diagrams by saying that two
virtual strings are related if there is a finite sequence of these
flattened moves tranforming one diagram to the other. This relation is
an equivalence relation.
Following Carter, Kamada and Saito \cite{CKS:StableEquivalence},
Kadokami showed that the set of virtual strings is equivalent to the set
of equivalence classes of virtual string diagrams under this equivalence
relation \cite{Kadokami:Non-triviality}.
This equivalence relation is also called \emph{homotopy}.
\par
The \emph{crossing number} of a \surfdiag{} is the number of crossings
appearing 
in it. We define the \emph{minimal crossing number} of a \surfdiag{}
$(S,D)$ to be the minimum crossing number of all \surfdiag s that are
homotopic to $(S,D)$. Clearly this is a virtual string
invariant.
\section{Nanowords}\label{sec:nanowords}
Turaev defined the concept of a nanoword in
\cite{Turaev:Words}. He showed how nanowords can be used to represent virtual
strings in \cite{Turaev:KnotsAndWords}. In general we can use nanowords
to represent other kinds of objects such as virtual knots but in this
paper we will always use the term nanoword to mean a nanoword
representing a virtual string.
\par
A \emph{nanoword} is a pair $(w,\pi)$ where $w$ is a Gauss word and $\pi$ is a
map from the letters appearing in $w$ to the set $\{a,b\}$. Recall that
a Gauss word is a finite sequence of letters where each letter that
appears, appears exactly twice. We follow Turaev and write $\xType{X}$
for $\pi(X)$ for a letter $X$ appearing in the Gauss word $w$.
\par
Given a \surfdiag{} we construct a nanoword to represent
it. We first pick some non-double point as a base
point and label all the crossings. Starting at the base point we
follow the curve according to its orientation. We read off the label of
each crossing as we pass through it, stopping when we get back to
the base point for the first time. The result is a Gauss word.
For example, the Gauss word associated with the diagram
in Figure~\ref{fig:diagram31} with base point $O$ is $ABCBAC$.
\par
\begin{figure}[hbt]
\begin{center}
\includegraphics{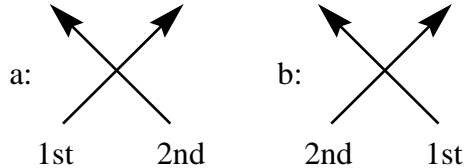}
\caption{The two types of crossing}
\label{fig:crossing}
\end{center}
\end{figure}
Because we have a base point, we can classify crossings into types
depending on the direction which the second arc going through the 
crossing crosses the first. This is shown in Figure~\ref{fig:crossing}.
In the diagram in Figure~\ref{fig:diagram31} the types of $A$ and $B$
are $a$ and the type of $C$ is $b$. This defines a map $\pi$ and we
write $\xType{A}=\xType{B}=a$ and $\xType{C}=b$.
\par
The Gauss word and the map together give a nanoword representing the
diagram. We can write the nanoword compactly as $\nanoword{ABCBAC}{aab}$ where
the types are listed in alphabetical order of the letters in the Gauss
word. In other words
$\nanoword{ABCBAC}{aab} = \nanoword{ABCBAC}{\xType{A}\xType{B}\xType{C}}$.
Note that a \surfdiag{} with no double points has an empty Gauss word
and the map $\pi$ is a map from the empty set to $\{a,b\}$. We write
this nanoword compactly as $\trivial$.
\par
An isomorphism between two nanowords $(w_1,\pi_1)$ and $(w_2,\pi_2)$
is a bijection $i$ between the sets of letters appearing in $w_1$ and $w_2$
such that $i$ maps the $n$th letter of $w_1$ to the $n$th letter of
$w_2$ for all $n$ and for each letter $X$ in $w_1$,
$\pi_1(X)$ is equal to $\pi_2(X)$.
Informally, an isomorphism is just a relabelling of the crossings.
\par
Turaev defined the following moves for nanowords.
In these move descriptions, the upper case
letters $A$, $B$ and $C$ represent arbitrary individual letters.
The lower case
letters $x$, $y$, $z$ and $t$ represent arbitrary sequences of letters
such that both sides of each move are Gauss words.
\par
Shift move:
\begin{equation*}
AxAy \longleftrightarrow xByB
\end{equation*}
where $A$ and $B$ map to opposite types.
\par
Homotopy move 1 (H1):
\begin{equation*}
xAAy \longleftrightarrow xy
\end{equation*}
where $A$ maps to $a$ or $b$.
\par
Homotopy move 2 (H2):
\begin{equation*}
xAByBAz \longleftrightarrow xyz
\end{equation*}
where $A$ and $B$ map to opposite types.
\par
Homotopy move 3 (H3):
\begin{equation*}
xAByACzBCt \longleftrightarrow xBAyCAzCBt
\end{equation*}
where $A$, $B$ and $C$ all map to the same type.
\par
Turaev derived some other moves from the moves H1, H2 and H3. 
We quote the moves here. The proofs appear in Lemmas 3.2.1 and 3.2.2 in
\cite{Turaev:Words}. 
\par
Homotopy move 3a (H3a)
\begin{equation*}
xAByCAzBCt \longleftrightarrow xBAyACzCBt
\end{equation*}
where $A$ and $C$ map to the same type and $B$ maps to the opposite
type.
\par
Homotopy move 3b (H3b)
\begin{equation*}
xAByCAzCBt \longleftrightarrow xBAyACzBCt
\end{equation*}
where $A$ and $B$ map to the same type and $C$ maps to the opposite
type.
\par
Homotopy move 3c (H3c)
\begin{equation*}
xAByACzCBt \longleftrightarrow xBAyCAzBCt
\end{equation*}
where $B$ and $C$ map to the same type and $A$ maps to the opposite
type.
\par
Homotopy move 2a (H2a)
\begin{equation*}
xAByABz \longleftrightarrow xyz
\end{equation*}
where $A$ and $B$ map to opposite types.
\begin{rem}
The set of homotopy moves on nanowords corresponds to the flat
Reidemeister moves and all possible variants of those moves derived by
changing orientations of arcs. 
See \cite{Gibson:mthesis} for a detailed description of the
correspondence between these sets of moves.
\end{rem}
As for moves on \surfdiag s, we collectively refer to any
of the moves given above involving $1$, $2$ or $3$ letters as
$1$-moves, $2$-moves or $3$-moves respectively. 
\par
Two nanowords are said to be homotopic if there exists a finite
sequence of the moves H1, H2 and H3, shift moves and
isotopies starting at one nanoword and finishing at the other.
This defines an equivalence relation called homotopy. Turaev proved that
there is a bijection between the equivalence classes of the set of
nanowords under this relation and the set of virtual strings
\cite{Turaev:KnotsAndWords}. In particular this means that homotopy
invariants of nanowords are invariants of virtual strings.
\section{Invariants}\label{sec:invariants}
In this section we briefly describe three invariants of virtual strings
defined by Turaev in \cite{Turaev:2004}, the $u$-polynomial, primitive
based matrices and coverings.
\par
In a given nanoword $\alpha$ we define the linking number of two distinct letters
$A$ and $B$ as follows. If $A$ and $B$ alternate in $\alpha$ (that is
$\alpha$ has the form $uAvBxAyBz$ or $uBvAxByAz$) then $A$ and $B$ are
said to be linked. In any other cases $A$ and $B$ are said to be
unlinked. When $A$ and $B$ are linked, we can use the shift move
to transform the nanoword $\alpha$ into the form $uAvBxAyBz$, where
$\xType{A}$ is $a$. Then, if $\xType{B}$ is $a$ we say that $B$ links
$A$ positively and if $\xType{B}$ is $b$ we say that $B$ links
$A$ negatively.
We define $\link{A}{B}$ as $0$ if $A$ and $B$ are unlinked,
$1$ if $B$ links $A$ positively and $-1$ if $B$ links $A$ negatively.
We define $\link{X}{X}$ to be $0$ for all letters $X$ in $\alpha$.
\par
Now for a letter $X$ in $\alpha$ we define
\begin{equation*}
\n(X) = \sum_{Y \in \alpha}\link{X}{Y}.
\end{equation*}
Next, for $k$ an integer greater than or equal to one, we define
\begin{equation*}
u_k(\alpha) = \# \lbrace X\in \alpha | \n(X)=k \rbrace - \# \lbrace X\in \alpha | \n(X)=-k \rbrace
\end{equation*}
where $\#$ indicates the number of elements in the set.
We then define the $u$-polynomial of $\alpha$ as
\begin{equation*}
u_{\alpha}(t) = \sum_{k \geq 1}u_k(\alpha)t^k.
\end{equation*}
Turaev showed that the $u$-polynomial is a homotopy invariant of $\alpha$.
\par
The second invariant we consider is the primitive based matrix
invariant.
\par
Given a virtual string $\Gamma$, we pick a \surfdiag{} $(S,D)$ which
represents it. We label the crossings appearing in $(S,D)$
and define $G$ to be the set of crossing labels union a special
element $s$. For each element $g$ in $G$ we define a curve which we
denote $g_c$.
\par
We define $s_c$ to be the whole curve in $(S,D)$. Any other element $g$ in
$G$ corresponds to a crossing $X_g$ in $(S,D)$. We can always orient the
crossing so that it appears as in the left of
Figure~\ref{fig:virtualcrossings}. We then define $g_c$ to be a curve on
$S$ 
parallel to the curve starting at $X_g$, leaving on the right hand outgoing
arc, and returning to $X_g$ on the right hand incoming arc. 
\par
We define a map $b$ from $G\times G$ to $\Z$ by using the curves we have
defined. We define $b(g,h)$ to be equal to the number of real crossings for which
$h_c$ crosses $g_c$ from right to left minus the number of real
crossings for which $h_c$ crosses $g_c$ from left to right. 
This is just the homological intersection number of $g_c$ with $h_c$. 
By the anti-symmetry in the definition it
follows that $b$ is skew-symmetric, that is $b(g,h) = -b(h,g)$ for all
$g$ and $h$ in $G$. The \emph{based matrix} of $(S,D)$ is defined to be
the triple $(G,s,b)$.
\par
Two based matrices $(G_1,s_1,b_1)$ and $(G_2,s_2,b_2)$ are
said to be \emph{isomorphic} if there is a bijective map $f$ from $G_1$
to $G_2$ which maps $s_1$ to $s_2$ and for which $b_1(g,h)$ equals
$b_2(f(g),f(h))$ for all $g$ and $h$ in $G_1$. 
\par
Turaev defined some moves on based matrices which allow us to derive
smaller based matrices by removing one or two elements under specific
conditions. A based matrix for which no moves are available is called
\emph{primitive}. By applying moves to the based matrix of
\surfdiag{} $(S,D)$, we can derive a primitive based matrix. Turaev proved
that, up to isomorphism, this primitive based matrix is a homotopy
invariant of the virtual string $\Gamma$ which $(S,D)$ represents
\cite{Turaev:2004}.
\par
Any invariant of the primitive based matrix of a virtual string $\Gamma$
is thus an invariant of $\Gamma$.
The simplest example of such an invariant is the size of the primitive
based matrix. None of the moves on based matrices allows us to remove
the special element $s$, and so all based matrices have size greater than or
equal to one. We define $\rho(\Gamma)$ to be the size of the primitive
based matrix of $\Gamma$ minus $1$. In \cite{Turaev:2004}, Turaev
suggested some other invariants of primitive based matrices. 
\par
Turaev showed that the primitive based matrix of a virtual string
determines the $u$-polynomial of the virtual string. In fact, the
primitive based matrix invariant is stronger than the
$u$-polynomial. Turaev showed this in \cite{Turaev:2004} and it can also
be seen in Section~\ref{sec:results}.
\par
The third invariant we consider is a covering of a virtual string.
For a nanoword $\alpha$ and a positive integer $r$ we construct a 
new nanoword $\cover{\alpha}{r}$ by deleting all letters $X$ in $\alpha$
where $\n(X)$ is not equal to $kr$ for some $k$ in $\Z$. We call
$\cover{\alpha}{r}$ the \textcover{r} of $\alpha$. 
In \cite{Turaev:2004}, Turaev showed that if $\alpha_1$
and $\alpha_2$ are homotopic nanowords then $\cover{\alpha_1}{r}$ and
$\cover{\alpha_2}{r}$ must also be homotopic. This means that the
\textcover{r} of a virtual string $\Gamma$ is well-defined and we write
it as $\cover{\Gamma}{r}$. If $\alpha$ represents $\Gamma$,
$\cover{\alpha}{r}$ represents $\cover{\Gamma}{r}$.
\begin{rem}
When $r$ is $1$, $\cover{\Gamma}{1}$ is equal to $\Gamma$ for all
virtual strings $\Gamma$. Thus in this case the operation of covering is
 just the identity operation.
\end{rem}
\section{Canonical description of primitive based matrices}\label{sec:canonical}
To effectively use primitive based matrices as an invariant of virtual
strings we need an easy way to determine whether two primitive based
matrices, $P_1$ and $P_2$, of size $n$ are isomorphic or not. In
specific cases this may just be a simple case of observing that $P_1$
contains some integer $x$ which does not appear in $P_2$. However, in
general, a systematic approach will be more useful.
\par
One such approach would consist of testing all $\factorial{(n-1)}$
different bijections between the sets associated with $P_1$ and
$P_2$ that map the special element in $P_1$ to the special element in
$P_2$, to see whether we have an isomorphism. One downside of this
approach is that as $n$ gets larger the number of bijections we have to
consider increases exponentially. Another downside is that we have to
repeat the process each time we want to compare $P_2$ to another based
matrix. 
\par
The approach we have adopted is to find a canonical description for each
based matrix. Two primitive based matrices are equivalent under
isomorphism if and only if their canonical descriptions are the
same. This description is easy to calculate by computer and two such
descriptions are easy to compare even by hand. This section explains how
the canonical description is calculated.
\par
Let $\basedMatrices{n}$ denote the set of equivalence classes of based
matrices of size $n$ under isomorphism. We will define an injective map
$\phi$ from $\basedMatrices{n}$ to $\Z^k$ where $k$ is
$\frac{1}{2}n(n-1)$.
\par
For a based matrix $(G,s,b)$, we can pick an ordering of the
elements of $G$ and then write $b$ as a skew-symmetric matrix. When we
write $b$ in this way, by convention we always pick the special element
$s$ to be first in the ordering.
Thus any based matrix can be written as a skew-symmetric matrix in up to
$\factorial{(n-1)}$ different ways.
\par
We can define a map $\theta$ from the set of
$n$ by $n$ skew-symmetric matrices $\skewSymMatrices{n}$ to $\Z^k$ as
follows. Given a skew-symmetric matrix $A$ with entries $a_{i,j}$ we map
$A$ to the $k$-tuple
\begin{equation*}
(a_{2,1},a_{3,1},\dotsc,a_{n,1},a_{3,2},a_{4,2},\dotsc,a_{n,2},a_{4,3},\dotsc,a_{n,n-1})
\end{equation*}
where we have listed the entries in the lower left triangular area, below
the main diagonal. We list the entries in each column from top to
bottom, starting from the left column. For example, if $M$ is the
skew-symmetric matrix
\begin{equation*}
\begin{pmatrix}
 0 & -1 & 2 & -1 \\
 1 &  0 & 2 &  0 \\
-2 & -2 & 0 & -1 \\
 1 &  0 & 1 &  0 \\
\end{pmatrix},
\end{equation*}
then $\theta(M)$ is $(1, -2, 1, -2, 0, 1)$.
\par
Note that given an element $p$ in $\Z^k$ we can use skew-symmetry to
construct a unique skew-symmetric matrix $A$ such that $\theta(A)$ is
$p$. It is then easy to see that $\theta$ is a bijection.
\par
Standard numerical order on $\Z$ induces an order on $\Z^k$ as
follows. Assume $p$ and $q$ are different elements in $\Z^k$. We write
$p$ as $(p_1,\dotsc,p_k)$ and $q$ as $(q_1,\dotsc,q_k)$. We say $p$ is
less than $q$ if there exists some $i$ less than or equal to $k$ such that
$p_i$ is less than $q_i$ and for all $j$ less than $i$, $p_j$ and $q_j$
are equal.
\par
To get a canonical description of a based matrix we could just consider
all $\factorial{(n-1)}$ associated skew-symmetric matrices, take
$\theta$ of each matrix and then take the minimal value in
$\Z^k$. However, this means that we must consider $\factorial{(n-1)}$
matrices for each based matrix. We take a different approach which, at
the cost of some complexity, tries to reduce the number of matrices we
have to consider. Of course, the canonical description we output depends
on the approach that we take.
\par
Our strategy is to break up the elements of the set associated with
the based matrix $P$ into equivalence classes, invariant under
isomorphism, and then order those equivalence classes in a way that is
also invariant under isomorphism. Then we consider all possible
permutations of 
elements within each equivalence class for all of the equivalence
classes. If there are $l$ such classes and the numbers of 
elements in each class are given by $n_1, n_2, \dotsc, n_l$ then the
total number of cases we consider is
\begin{equation*}
\prod_{i=1}^{l}\factorial{n_i}
\end{equation*}
which is bounded above by $\factorial{(n-1)}$.
In each case we get a skew-symmetric matrix $A$ and we can calculate
$\theta(A)$. We then take the minimal value of $\theta(A)$ in $\Z^k$
over all the cases and set this to be $\phi(P)$. As the equivalence
classes and the order on them are invariant under isomorphism and we
take the minimal value over all possible cases, it is clear that
$\phi(P)$ is injective. That is, if $\phi(P_1)$ equals $\phi(P_2)$,
$P_1$ and $P_2$ must be isomorphic as based matrices.  
\par
We note that in the best case we may end up with equivalence classes
each only containing a single element. In the worst case we may just
have a single equivalence class containing all the elements.
\par
We now define an equivalence relation on the elements in the set
associated with a based matrix $(G,s,b)$. As by convention the
element $s$ always comes first in the matrix we just need to consider
the other elements. 
\par
We use two properties of an element $g$ in $G$ which
are invariant under isomorphism.
The first property is simply the value $b(g,s)$. Considering an
isomorphism $f$ from $(G,s,b)$ to
$(G^\prime,s^\prime,b^\prime)$, we have
\begin{equation*}
b(g,s) = b^\prime(f(g),f(s)) = b^\prime(f(g),s^\prime),
\end{equation*}
which shows invariance under isomorphism.
\par
The second is a map $m_g$ from $\Z$ to $\Znonneg$. It is defined as follows
\begin{equation*}
m_g(i) = \#\{ h \in G-\{s\} | b(g,h) = i \}
\end{equation*}
where $\#$ indicates the number of elements in the set.
We check how $m_g$ behaves under the isomorphism $f$ defined above. As
$f$ is a bijection, for each $h$ in $G^\prime-\{s^\prime\}$ there exists
exactly one $k$ in $G-\{s\}$ such that $f(k)$ is $h$. By the definition
of isomorphism we also have $b(g,k) = b^\prime(f(g),f(k))$. Putting these
together we get 
\begin{align*}
m_{f(g)}(i) = & \#\{ h \in G^\prime-\{s^\prime\} | b^\prime(f(g),h) = i \} \\
= & \#\{ k \in G-\{s\} | b^\prime(f(g),f(k)) = i \} \\
= & \#\{ k \in G-\{s\} | b(g,k) = i \} \\
= & m_g(i).
\end{align*}
This shows that $m_g$ is invariant under isomorphism.
\par
We note there are a finite number of non-zero values of $m_g(i)$. Say
there are $l$ of them. Then we can represent $m_g$ as
$l$ pairs of the form $(i,m_g(i))$. We use the first element in each
pair to sort the pairs so that $i$ increases as we go through the
list. By concatenating the pairs into a $2l$-tuple we can summarise
$m_g$ uniquely. 
\par
As an example, suppose $m_g(2)=1$, $m_g(-1)=2$, $m_g(0)=3$ and
$m_g(i)=0$ for all other values of $i$. Then just considering the
non-zero values we get the pairs $(2,1)$, $(-1,2)$ and
$(0,3)$. Sorting these we get $(-1,2)$, $(0,3)$ and
$(2,1)$. Concatenating gives the $6$-tuple
$(-1,2,0,3,2,1)$.
\par
If $p=(p_i)$ is a $k$-tuple and $q=(q_i)$ is an $l$-tuple, we say $p$ is
less than $q$ if there exists a $j$ such that $p_j$ is less than $q_j$
and $p_i$ equals $q_i$ for all $i$ less than $j$, or $p_i$ equals $q_i$
for all $i$ less than or equal to $k$ and $k$ is less than $l$. 
We can then define an ordering on the maps $m_g$ by saying $m_g$ is less
than $m_h$ if the tuple associated with $m_g$ is less than the tuple
associated with $m_h$.
\par
We define our equivalence relation on $G-\{s\}$ by saying that $g$ and $h$
are equivalent if and only if $b(g,s)$ is equal to $b(h,s)$ and $m_g(i)$
is equal to $m_h(i)$ for all $i$. We write $[g]$ for the equivalence
class of $g$ under this relation.
We can then define an ordering on the equivalence classes in $G-\{s\}$ by
saying that $[g]$ is less than $[h]$ if $b(g,s)$ is less than $b(h,s)$,
or if $b(g,s)$ is equal to $b(h,s)$ and $m_g$ is less than $m_h$.
\par
To conclude this section we calculate $\phi$ of a based matrix as an
example.
\begin{ex}
Let $(G,s,b)$ be a based matrix, where $G$ is $\{s,A,B,C\}$ and the
 table 
\begin{equation*}
\begin{tabular}{ccccc}
    &  $s$ &  $A$ & $B$ &  $C$ \\
$s$ &  $0$ & $-1$ & $2$ & $-1$ \\
$A$ &  $1$ &  $0$ & $2$ &  $0$ \\
$B$ & $-2$ & $-2$ & $0$ & $-1$ \\
$C$ &  $1$ &  $0$ & $1$ &  $0$ \\
\end{tabular}
\end{equation*}
defines $b$.
\par
As $b(B,s)$ is $-2$ and $b(A,s)$ and $b(C,s)$ are both $1$, $B$ is
 in a different equivalence class to $A$ and $C$. In particular $[B]$ is
 less than $[A]$ and $[C]$.
\par
The map $m_A$ is given by $m_A(0)=2$, $m_A(2)=1$ and $m_A(i)=0$ for all
 other $i$. We summarise this as the $4$-tuple $(0,2,2,1)$. 
The map $m_C$ is given by $m_C(0)=2$, $m_C(1)=1$ and $m_A(i)=0$ for all
 other $i$. We summarise this as the $4$-tuple $(0,2,1,1)$. Comparing
 the two tuples we see that the one for $C$ is less than the one for
 $A$ and so $[C]$ is less than $[A]$.
\par
In this case we are able to order the letters completely. The order we
 get is $s$, $B$, $C$, $A$. The corresponding matrix is
\begin{equation*}
\begin{pmatrix}
 0 & 2 & -1 & -1 \\
-2 & 0 & -1 & -2 \\
 1 & 1 &  0 &  0 \\
 1 & 2 &  0 &  0 \\
\end{pmatrix},
\end{equation*}
and so $\phi((G,s,b))$ is $(-2, 1, 1, 1, 2, 0)$.
\end{ex}
\section{Algorithm for enumeration}\label{sec:enumeration}
For any given non-negative integer $n$ our aim is to enumerate all
distinct virtual strings with minimal crossing number equal to $n$. We
give an algorithm to do this by using nanowords. Before explaining
the algorithm we make some definitions.
\par
An \emph{alphabet} $\mathcal{A}$ is a finite ordered set.
We call the elements of $\mathcal{A}$ \emph{letters}. We call the
ordering on $\mathcal{A}$ \emph{alphabetical order}. For use in examples
below we define $\mathcal{B}$ to be the $3$ letter alphabet $\{A,B,C\}$
with the standard alphabetical order.
\par
A Gauss word on an alphabet $\mathcal{A}$ is a sequence of letters in
$\mathcal{A}$ where each letter of $\mathcal{A}$ appears exactly
twice. If $\mathcal{A}$ has $n$ letters, a Gauss word on $\mathcal{A}$
necessarily has $2n$ letters. 
For example, $BBAA$ and $CACBCA$ are not Gauss words on $\mathcal{B}$,
but $BCACBA$ is.
\par
Note that the ordering on $\mathcal{A}$ induces an ordering on Gauss
words on $\mathcal{A}$. Given two Gauss words on $\mathcal{A}$, $u$ and
$v$, we compare initial sequences of letters in the words until we
find the first pair of letters that differ. The order of $u$ and $v$ are
determined by the order of the differing letters. If there is no such
differing pair, $u$ and $v$ are the same Gauss word. 
As with the ordering on the letters in $\mathcal{A}$, we call this
ordering on Gauss words alphabetical order.
As an example, $BCACBA$ comes before $BCCABA$ and after $BCACAB$ in the
alphabetical order induced by $\mathcal{B}$.
\par
Two Gauss words $w$ and $x$ on $\mathcal{A}$ are said to be $isomorphic$
if there exists a bijection $f$ from $\mathcal{A}$ to itself such that
$x$ is the result of applying $f$ letterwise to $w$. For example
$CBAACB$ is isomorphic to $ABCCAB$ by the bijection mapping $C$ to $A$,
$B$ to $B$ and $A$ to $C$. It is clear that isomorphism defines an
equivalence relation on the set of Gauss words on a fixed alphabet
$\mathcal{A}$.
\par
We say that a Gauss word $w$ on $\mathcal{A}$ is \emph{increasing} if
the first occurence of each letter of $\mathcal{A}$ in $w$ appears in
alphabetical order. For example, on $\mathcal{B}$ the Gauss word
$ABACBC$ is increasing and the Gauss word $CACBAB$ is not.
\begin{lem}\label{lem:gauss_increasing}
Given an alphabet $\mathcal{A}$, every Gauss word on $\mathcal{A}$ is
 isomorphic to an increasing Gauss word on $\mathcal{A}$.
\end{lem}
\begin{proof}
Assume $\mathcal{A}$ has $n$ letters. We write $N$ for the set
 $\{1,2,\dotsc,n\}$. Then the order on 
 $\mathcal{A}$ gives a bijection $o$ which maps $N$ to $\mathcal{A}$.
For example, for $\mathcal{B}$, $o(1)=A$, $o(2)=B$ and $o(3)=C$.
\par
Given a Gauss word $w$ on $\mathcal{A}$, we can define a map $p$ from
 $N$ to $\mathcal{A}$ by defining $p(i)$ to be the $i$th new letter in
 $w$. For example, with $\mathcal{B}$ as above, $p$ for the Gauss word
 $CACBAB$ is given by $p(1)=C$, $p(2)=A$ and $p(3)=B$.
\par
As a Gauss word on $\mathcal{A}$ contains all letters in $\mathcal{A}$,
 $p$ must be surjective. As $p(i)$ is the $i$th new letter in the Gauss
 word, $p(i)=p(j)$ implies $i=j$ and so $p$ is injective. Thus $p$ is
 bijective.
\par
We define a map $f$ from $\mathcal{A}$ to $\mathcal{A}$ by setting
\begin{equation*}
f(X) = o(p^{-1}(X)). 
\end{equation*}
As $o$ and $p$ are bijections, so is $f$. We can thus apply $f$
 letterwise to the Gauss word $w$ to get an isomorphic Gauss word
 $w^{\prime}$. In $w^{\prime}$ the $i$th new letter will be 
\begin{equation*}
f(p(i)) = o(p^{-1}(p(i))) = o(i).
\end{equation*}
Thus the first occurence of each letter in $w^{\prime}$ appears in
 alphabetical order and so $w^{\prime}$ is increasing.
\end{proof}
\par
A nanoword on $\mathcal{A}$ is a Gauss word on $\mathcal{A}$ with a map
from $\mathcal{A}$ to $\{a,b\}$.
By Lemma~\ref{lem:gauss_increasing} we need only deal with
nanowords that have
increasing Gauss words. Whenever we apply a shift move or a homotopy
move to a nanoword $\alpha$ to get a nanoword $\beta$ that does not have an
increasing Gauss word, we can apply an isomorphism to $\beta$ to get a
nanoword with an increasing Gauss word. 
\par
As $\mathcal{A}$ contains a finite number of letters, the set of
nanowords on $\mathcal{A}$ up to isomorphism is finite. We give an
explicit calculation of its size here. If $\mathcal{A}$
has $n$ letters, the number of different Gauss words on $\mathcal{A}$ is
given by
\begin{equation*}
\frac{\factorial{(2n)}}{2^n}.
\end{equation*}
We note that each equivalence class of Gauss words on $\mathcal{A}$
under isomorphism contains $\factorial{n}$ words. This is because the
order of the first occurence of each letter of $\mathcal{A}$ in the
Gauss word determines the Gauss word within the equivalence class.
Thus the number of increasing Gauss words on $\mathcal{A}$ is
\begin{equation*}
\frac{\factorial{(2n)}}{\factorial{n}2^n}.
\end{equation*}
The number of nanowords on $\mathcal{A}$ with a given Gauss word is
equal to the number of maps from $\mathcal{A}$ to $\{a,b\}$. This is
just $2^n$. Thus the total number of nanowords on $\mathcal{A}$ with
increasing Gauss words, and thus the total number of nanowords on
$\mathcal{A}$ up to isomorphism is
\begin{equation*}
\frac{\factorial{(2n)}}{\factorial{n}}.
\end{equation*}
\par
We define the \emph{3-class} of a nanoword $\alpha$ as the equivalence class
under shift moves and $3$-moves. We write $\left[\alpha\right]_3$ for the
$3$-class of $\alpha$. 
\par
A \emph{type word} is a sequence of letters in the set $\{a,b\}$.
By defining $a$ to be less
than $b$, we can define an order on type words in the same way as we
defined an order on Gauss words. Given a nanoword $\alpha$ on
$\mathcal{A}$, we consider each letter of $\mathcal{A}$ in alphabetical
order. For each letter we write the type of that letter in $\alpha$. The
concatenation of the resultant sequence of letters is defined to be the
type word of a nanoword $\alpha$ on $\mathcal{A}$. For example, the type
word of the nanoword $\nanoword{ABACBC}{aab}$ on $\mathcal{B}$ is $aab$.
\par
We define an order on nanowords on $\mathcal{A}$ in the following way.
We say that $\alpha$ is less than $\beta$ if its Gauss word is less than that
of $\beta$, or if the Gauss words of $\alpha$ and $\beta$ are equal and
the type word of $\alpha$ is less than that of $\beta$.
So, for example, $\nanoword{ABACBC}{bba}$ is less than
$\nanoword{ABCBCA}{aab}$ but greater than $\nanoword{ABACBC}{aab}$.  
\par
We say that a nanoword $\alpha$ is \emph{alphabetically minimal} in
$\left[\alpha\right]_3$ if there does not exist a nanoword $\beta$ less
than $\alpha$ such that $\left[\beta\right]_3$ is equal to
$\left[\alpha\right]_3$.
\par
We say that a nanoword is \emph{reducible} if there exists a crossing
reducing $1$-move or $2$-move which can be applied to the nanoword. We say
that the $3$-class of a nanoword $\alpha$ is \emph{reducible} if there
exists a reducible nanoword $\beta$ such that $\left[\beta\right]_3$ is
equal to $\left[\alpha\right]_3$. A nanoword or a $3$-class is
\emph{irreducible} if it is not reducible. 
\par
To generate all the virtual string candidates with $n$ crossings we first
pick some ordered set of $n$ elements $\mathcal{A}$. We could just
use the set of integers from $1$ to $n$. However, to display the nanowords
in compact form it is useful to be able to represent each letter in the
alphabet with a single character. Thus, when $n$ is less than or equal
to $26$, we use the initial $n$ letters of the $26$ letter English
alphabet and inherit the usual ordering.
\par
The generating algorithm works in the following way:
\par
1. Enumerate all increasing Gauss words on $\mathcal{A}$.
\par
2. For each Gauss word in the list we consider all possible assignments
of the crossing types $a$ and $b$ to the $n$ letters. This gives us
$2^n$ possible nanowords for each Gauss word.
\par
3. For each nanoword $\alpha$ in the list, we generate the set of elements in
equivalence class of $\alpha$ under shift and $3$-moves. 
\par
4. If $\alpha$ is not alphabetically minimal in the equivalence class, we
discard $\alpha$ and then consider the next nanoword in the list.
\par
5. If any nanoword in the equivalence class is reducible, we discard $\alpha$ and
then consider the next nanoword in the list.
\par
6. Add $\alpha$ to the list of virtual string candidates and then consider
the next nanoword in the list.
\par
We note that Step~4 is required to prevent duplicates appearing in the
list. Step~5 is required to check for minimality of the crossing
number. If $\alpha$ is homotopic to a reducible $n$ letter nanoword $\beta$
then $\alpha$ is homotopic to some nanoword $\beta^\prime$ with less than
$n$ crossings. Thus the minimal crossing number of the virtual string
represented by $\alpha$ is less than $n$.
\par
We note that there are ways to optimise the algorithm. We consider
two simple ways here.
\par
Firstly we can reduce the number of nanowords that we have to consider. Let
$w$ be a Gauss word of the form $yXXz$ for some letter $X$ and some
(possibly empty) sequences of letters $y$ and $z$. When we derive a nanoword
$\alpha$ from $w$ by assigning crossing types, no matter whether we
assign type $a$ or type $b$ to $X$, $\alpha$ will be reducible by a
$1$-move. We can either avoid generating such Gauss words, or eliminate
them as they are generated during Step~1.
\par
Secondly we can make the checks in Step~4 and Step~5 as we generate the
set of elements in the equivalence class of a nanoword $\alpha$ in Step~3. If we
generate a nanoword $\beta$ that is alphabetically less than $\alpha$ or is
reducible, we can discard $\alpha$ straight away. There is no need to
continue calculating the $3$-class.
\par
Once we have a list of virtual string candidates for minimal crossing
number $n$ we need to determine whether they are distinct from each
other and whether they are distinct from virtual string candidates of
smaller minimal crossing number. One approach is to use virtual string
invariants and we consider this in the next section.
\section{Results of enumeration}\label{sec:results}
It is clear that there is a unique virtual string with $0$
crossings. Methodical examination of nanowords of $1$ and $2$ letters
shows that they are all homotopic to the trivial nanoword. Thus
there are no virtual strings with minimum crossing number of $1$ or
$2$. For $3$ letters we use the algorithm to get the following two
nanowords $\nanoword{ABACBC}{aab}$ 
and $\nanoword{ABACBC}{abb}$. We can calculate the $u$-polynomials of these
nanowords. For the trivial virtual string the $u$-polynomial is $0$. For
$\nanoword{ABACBC}{aab}$, the $u$-polynomial is $2t-t^2$. For
$\nanoword{ABACBC}{abb}$, the $u$-polynomial is $t^2-2t$. Thus these three
nanowords are all mutually non-homotopic. Thus the number of virtual strings
with minimum crossing number $3$ is $2$. Turaev has already pointed out
this fact in \cite{Turaev:2004}.
\par
When we consider virtual strings with minimum crossing number of $4$ we
find examples of strings that cannot be distinguished by the
$u$-polynomial alone. For example, the nanowords $\nanoword{ABACBDCD}{abab}$
and $\nanoword{ABCADCBD}{aaab}$ have $u$-polynomial $0$ which is the same as
the trivial virtual string. However, the primitive based matrices of
these two nanowords are different to each other and different to the
primitive based matrix of the trivial virtual string. This implies that
the virtual strings they represent are all distinct.
\par
For minimum crossing number of $4$ we used a computer to find $26$
virtual string candidates. We can show that these are all distinct from
each other and from the $3$ virtual strings of lower minimum crossing
number by using primitive based matrices. We show the results in
Table~\ref{tab:vstrings4}. As far as we are aware, the total of $26$ for
minimum crossing number of $4$ was previously unknown.
We summarise the number of distinct virtual strings for minimum crossing
number of $4$ or less in Table~\ref{tab:numvstrings}.
\par
\begin{table}[hbt]
\begin{center}
\begin{tabular}{c|c|c|c|c}
ID & Nanoword & $u(t)$ & $\rho$ & Based matrix \\
\hline
\vstring{0}{1}  & \trivial & $0$ & $0$ & $0$ \\
\vstring{3}{1}  & ABACBC:aab    & $-t^2+2t$ & $3$ & $-2,1,1,1,2,0$ \\
\vstring{3}{2}  & ABACBC:abb    & $t^2-2t$ & $3$ & $-1,-1,2,0,2,1$ \\
\vstring{4}{1}  & ABABCDCD:aaaa & $0$ & $4$ & $-1,-1,1,1,0,1,0,0,1,0$ \\
\vstring{4}{2}  & ABABCDCD:aabb & $0$ & $4$ & $-1,-1,1,1,-1,1,1,1,1,1$ \\
\vstring{4}{3}  & ABABCDCD:bbbb & $0$ & $4$ & $-1,-1,1,1,0,1,2,2,1,0$ \\
\vstring{4}{4}  & ABACBDCD:aaaa & $0$ & $4$ & $-1,0,0,1,1,1,0,-1,1,1$ \\
\vstring{4}{5}  & ABACBDCD:aaab & $-t^2+2t$ & $4$ & $-2,0,1,1,1,2,1,1,1,1$ \\
\vstring{4}{6}  & ABACBDCD:aaba & $-t^2+2t$ & $4$ & $-2,0,1,1,2,1,1,1,0,0$ \\
\vstring{4}{7}  & ABACBDCD:abaa & $t^2-2t$ & $4$ & $-1,-1,0,2,0,0,1,1,1,2$ \\
\vstring{4}{8}  & ABACBDCD:abab & $0$ & $4$ & $-1,0,0,1,1,1,1,0,0,0$ \\
\vstring{4}{9}  & ABACBDCD:abba & $0$ & $4$ & $-2,-1,1,2,0,1,3,0,1,0$ \\
\vstring{4}{10} & ABACBDCD:abbb & $t^2-2t$ & $4$ & $-1,-1,0,2,-1,0,2,0,1,1$ \\
\vstring{4}{11} & ABACBDCD:baaa & $t^2-2t$ & $4$ & $-1,-1,0,2,-1,1,2,1,1,1$ \\
\vstring{4}{12} & ABACBDCD:baab & $0$ & $4$ & $-2,-1,1,2,1,2,1,2,2,1$ \\
\vstring{4}{13} & ABACBDCD:baba & $0$ & $4$ & $-1,0,0,1,0,0,1,0,1,1$ \\
\vstring{4}{14} & ABACBDCD:babb & $t^2-2t$ & $4$ & $-1,-1,0,2,0,1,2,0,2,0$ \\
\vstring{4}{15} & ABACBDCD:bbab & $-t^2+2t$ & $4$ & $-2,0,1,1,0,2,2,1,0,0$ \\
\vstring{4}{16} & ABACBDCD:bbba & $-t^2+2t$ & $4$ & $-2,0,1,1,1,2,1,0,0,1$ \\
\vstring{4}{17} & ABACBDCD:bbbb & $0$ & $4$ & $-1,0,0,1,0,0,2,1,0,0$ \\
\vstring{4}{18} & ABACDBCD:aabb & $-t^3+t^2+t$ & $4$ & $-3,0,1,2,2,1,3,0,1,0$ \\
\vstring{4}{19} & ABACDBCD:abbb & $t^3-t^2-t$ & $4$ & $-2,-1,0,3,0,1,3,0,1,2$ \\
\vstring{4}{20} & ABACDBCD:baaa & $t^3-t^2-t$ & $4$ & $-2,-1,0,3,1,1,2,1,3,1$ \\
\vstring{4}{21} & ABACDBCD:bbaa & $-t^3+t^2+t$ & $4$ & $-3,0,1,2,1,3,2,1,1,1$ \\
\vstring{4}{22} & ABACDBDC:aabb & $-t^3+3t$ & $4$ & $-3,1,1,1,1,2,3,0,0,0$ \\
\vstring{4}{23} & ABACDBDC:abbb & $t^3-3t$ & $4$ & $-1,-1,-1,3,0,0,3,0,2,1$ \\
\vstring{4}{24} & ABCADBCD:aaab & $-t^3+2t^2-t$ & $4$ & $-3,-1,2,2,1,2,3,1,2,0$ \\
\vstring{4}{25} & ABCADBCD:abbb & $t^3-2t^2+t$ & $4$ & $-2,-2,1,3,0,2,3,1,2,1$ \\
\vstring{4}{26} & ABCADCBD:aaab & $0$ & $4$ & $-2,-2,2,2,0,2,3,1,2,0$ \\
\end{tabular}
\end{center}
\caption{Virtual strings up to 4 crossings}
\label{tab:vstrings4}
\end{table}
\begin{table}[hbt]
\begin{center}
\begin{tabular}{c|c}
Crossings & Number \\
\hline
0 & 1 \\
1 & 0 \\
2 & 0 \\
3 & 2 \\
4 & 26
\end{tabular}
\end{center}
\caption{Numbers of virtual strings}
\label{tab:numvstrings}
\end{table}
\par
We can calculate the coverings of the virtual strings in
Table~\ref{tab:vstrings4}. All of the virtual strings in the table
except for 
one have trivial \textcover{r}s for all $r$ other than $1$. The
exception is \vstring{4}{26} which is fixed under
\textcover{2}. For $r$ not $1$ or $2$ its \textcover{r}s are also
trivial.
\par
Given a \surfdiag{} $(S,D)$, we can define a \surfdiag{} $(-S,D)$
where $-S$ is $S$ with the orientation reversed. We can think of
$(-S,D)$ as a mirror image of $(S,D)$.
It is easy to see that if $(S_1,D_1)$ and $(S_2,D_2)$ are homotopic
\surfdiag s then $(-S_1,D_1)$ is homotopic to $(-S_2,D_2)$.
Thus, for a virtual string $\Gamma$ represented by a \surfdiag{}
$(S,D)$, we can define the mirror of $\Gamma$ to be the virtual string
represented by $(-S,D)$. We write this virtual string
$\overline{\Gamma}$.
As swapping the orientation on a surface twice gets us back to the
original surface, this reflection operation
is an involution. Thus $\overline{\overline{\Gamma}}$ is $\Gamma$. 
\par
Similarly, given a \surfdiag{} $(S,D)$, we write $-D$ to denote $D$ with
its orientation reversed. We then define $(S,-D)$ to be the inverse of
$(S,D)$. 
Again, this operation is well-defined under homotopy. Thus, for a
virtual string $\Gamma$ represented by $(S,D)$, we can define the
inverse of $\Gamma$, $-\Gamma$, to be the virtual string represented by
$(S,-D)$. As the inverse of $(S,-D)$ is $(S,D)$, $-(-\Gamma)$ is
$\Gamma$.
\par
The two operations, reflection and inversion, are commutative. We can
compose these two operations to get the mirror inverse of $\Gamma$,
written $-\overline{\Gamma}$.
\par
Given a nanoword $\alpha$ representing a virtual string $\Gamma$ it is easy
to calculate nanowords representing $-\Gamma$, $\overline{\Gamma}$ and
$-\overline{\Gamma}$. Swapping the type of each letter in $\alpha$ derives a
new nanoword $\swap(\alpha)$ which represents $\overline{\Gamma}$.
We define $-\alpha$ to be the nanoword given by reversing the order
of the letters in the Gauss word of $\alpha$ and swapping the types of
the letters. Then $-\alpha$ represents $-\Gamma$
and $-\swap(\alpha)$ represents $-\overline{\Gamma}$.
\par
Usually, when tables of classical knots are given, reflections and
inversions of knots are considered to be the same knot type. It is
therefore natural to define an unoriented equivalence of virtual strings
where $\Gamma$, $-\Gamma$, $\overline{\Gamma}$ and $-\overline{\Gamma}$
are considered to be the same. Table~\ref{tab:symmvstrings} lists
distinct virtual strings under this unoriented equivalence and indicates
which virtual strings from Table~\ref{tab:vstrings4} are derived under
the operations of reflection, inversion and reflected inversion. Under
this kind of equivalence there is $1$ virtual string with no crossings,
$1$ virtual string with minimal crossing number $3$ and $11$ virtual
strings with minimal crossing number $4$.
\begin{table}[hbt]
\begin{center}
\begin{tabular}{c|c|c|c|c}
Virtual String & Mirror & Inverse & Mirror-Inverse & Symmetry Type \\
\hline
\trivial & = & = & = & a \\
\vstring{3}{1} & = & \vstring{3}{2} & \vstring{3}{2} & $+$ \\
\vstring{4}{1} & \vstring{4}{3} & \vstring{4}{3} & = & $-$ \\
\vstring{4}{2} & = & = & = & a \\
\vstring{4}{4} & \vstring{4}{17} & \vstring{4}{17} & = & $-$ \\
\vstring{4}{5} & \vstring{4}{16} & \vstring{4}{10} & \vstring{4}{11} & c \\
\vstring{4}{6} & \vstring{4}{15} & \vstring{4}{14} & \vstring{4}{7} & c \\
\vstring{4}{8} & \vstring{4}{13} & = & \vstring{4}{13} & i \\
\vstring{4}{9} & \vstring{4}{12} & \vstring{4}{12} & = & $-$ \\
\vstring{4}{18} & \vstring{4}{21} & \vstring{4}{20} & \vstring{4}{19} & c \\
\vstring{4}{22} & = & \vstring{4}{23} & \vstring{4}{23} & $+$ \\
\vstring{4}{24} & = & \vstring{4}{25} & \vstring{4}{25} & $+$ \\
\vstring{4}{26} & = & = & = & a \\
\end{tabular}
\end{center}
\caption{Virtual strings under reflection and inversion. Here `=' means
 the homotopy type of the virtual string is unchanged by the operation.}
\label{tab:symmvstrings}
\end{table}
\par
We note that the homotopy types of some virtual strings are unchanged
under some operations. 
We call a virtual string $\Gamma$ \emph{invertible} if $\Gamma$ is
homotopic to $-\Gamma$ and \emph{noninvertible} if it is not.
We call a virtual string $\Gamma$ \emph{amphicheiral} if $\Gamma$ is
homotopic to $\overline{\Gamma}$ or $-\overline{\Gamma}$ and
\emph{chiral} if it is not.
A virtual string $\Gamma$ which is homotopic to
$\overline{\Gamma}$ is called \emph{positive amphicheiral} (sometimes
written \emph{$+$amphicheiral}). 
A virtual string $\Gamma$ which is homotopic to
$-\overline{\Gamma}$ is called \emph{negative amphicheiral} (sometimes
written \emph{$-$amphicheiral}).
If a virtual string $\Gamma$ is amphicheiral and invertible we call it
\emph{fully amphicheiral}.
\par
We can thus classify virtual strings into
five different types depending on their behaviour under the three
operations, reflection, inversion and reflected inversion.
This classification is summarised in Table~\ref{tab:symmtypes}.
Note that the list is complete
because if two different operations do not change the homotopy type of
the virtual string, the third operation cannot either.
We remark that symmetries of classical knots also can be classified into
five different types depending on whether the knot is equivalent or not
to its mirror image, inversion or the mirror image of its
inversion. This is discussed in a paper by Hoste, Thistlethwaite and
Weeks \cite{Hoste/Thistlethwaite/Weeks:Knots}. We have used their
terminology to describe these symmetry types and
Table~\ref{tab:symmtypes} is based on a table from their paper. 
\par
Virtual strings of all
five types do actually exist. The final column in
Table~\ref{tab:symmvstrings} indicates the type of each virtual string. 
\begin{table}[hbt]
\begin{center}
\begin{tabular}{c|l|l}
Type & Unchanged Under & Description \\
\hline
c & None & Chiral, noninvertible \\
i & Inversion only & Chiral, invertible \\
$+$ & Reflection only & $+$Amphicheiral, noninvertible \\
$-$ & Inverted reflection only & $-$Amphicheiral, noninvertible \\
a & All & Amphicheiral, invertible
\end{tabular}
\end{center}
\caption{Classification of virtual strings by their behaviour under
 reflection, inversion and reflected inversion.}
\label{tab:symmtypes}
\end{table}
\par
When we consider virtual strings with minimum crossing number of $5$,
primitive based matrices are no longer enough to distinguish distinct
virtual strings.
We now give an example of a pair of virtual strings having minimum
crossing number $5$ which cannot be distinguished by their primitive
based matrices. We show that they are different by considering their
\textcover{2}s. 
\par
We write $\alpha$ for the the nanoword $\nanoword{ABACBDEDCE}{baabb}$ and
$\beta$ for the nanoword $\nanoword{ABACDECDBE}{bbaaa}$. The primitive based
matrices for these two nanowords are isomorphic.
As the $u$-polynomial is derived from the primitive based matrix, the
$u$-polynomials for $\alpha$ and $\beta$ are equal. In fact,
$u_\alpha(t)=u_\beta(t)=0$. 
\par
We will now calculate the \textcover{2} of $\alpha$ and $\beta$. In each
case we will remove letters $X$ for which $\n(X)$ is odd.
\par
For $\alpha$ we have $\n(A)=-1$, $\n(B)=2$, $\n(C)=-2$, $\n(D)=1$ and
$\n(E)=0$. We remove the letters $A$ and $D$ to get
$\nanoword{BCBECE}{aab}$ which represents $\cover{\alpha}{2}$. It is then
easy to calculate that the $u$-polynomial of $\cover{\alpha}{2}$ is
$-t^2+2t$. Thus $\cover{\alpha}{2}$ is non-trivial. In fact, since the
nanoword we obtained only has $3$ letters we can use
Table~\ref{tab:vstrings4} to identify $\cover{\alpha}{2}$ as the  
virtual string \vstring{3}{1}.
\par
For $\beta$ we have $\n(A)=1$, $\n(B)=-2$, $\n(C)=2$, $\n(D)=0$ and
$\n(E)=-1$. We remove the letters $A$ and $E$ to get
$\nanoword{BCDCDB}{baa}$ which represents $\cover{\beta}{2}$.
Calculating the $u$-polynomial of $\cover{\beta}{2}$, we find that it is
$0$. Using Table~\ref{tab:vstrings4} it is clear $\cover{\beta}{2}$
is trivial. Thus $\cover{\alpha}{2}$ and $\cover{\beta}{2}$ are not
homotopic and this implies that
$\alpha$ and $\beta$ are not homotopic either. 
\begin{table}[hbt]
\begin{center}
\begin{tabular}{c|c|c}
Nanoword & Based matrix & \textcover{2}\\
\hline
ABACBDECDE:abbbb & $-2,-1,0,1,2,-1,1,1,3,1,0,1,0,1,0$ & \trivial \\
ABACDECEBD:aaaba & $-2,-1,0,1,2,-1,1,1,3,1,0,1,0,1,0$ & \vstring{3}{2} \\
& & \\
ABACBDEDCE:aaaab & $-2,-1,0,1,2,0,1,1,3,0,0,1,1,1,-1$ & \vstring{3}{1} \\
ABACDECDBE:aabbb & $-2,-1,0,1,2,0,1,1,3,0,0,1,1,1,-1$ & \trivial \\
& & \\
ABACBDEDCE:baabb & $-2,-1,0,1,2,1,1,2,1,1,2,2,0,1,2$ & \vstring{3}{1} \\
ABACDECDBE:bbaaa & $-2,-1,0,1,2,1,1,2,1,1,2,2,0,1,2$ & \trivial \\
& & \\
ABACBDECDE:baaaa & $-2,-1,0,1,2,2,1,2,1,0,2,2,1,1,1$ & \trivial \\
ABACDECEBD:baabb & $-2,-1,0,1,2,2,1,2,1,0,2,2,1,1,1$ & \vstring{3}{2}
\end{tabular}
\end{center}
\caption{Pairs of virtual strings with 5 crossings that can be
 distinguished by coverings but not by their primitive based matrices.}
\label{tab:5onlybycover}
\end{table}
\par
In fact, for virtual strings of $5$ crossings, including the example
given above, there are four pairs of
virtual strings which have the same primitive based matrices but can be
distinguished by coverings. We list them in Table~\ref{tab:5onlybycover}.
\begin{table}[hbt]
\begin{center}
\begin{tabular}{c|c|c}
Nanoword & $\rho$ & Based matrix\\
\hline
ABABCDCD:aabb & 4 & $-1,-1,1,1,-1,1,1,1,1,1$ \\
ABABCDCEDE:bbaba & 4 & $-1,-1,1,1,-1,1,1,1,1,1$ \\
ABABCDCEDE:aabab & 4 & $-1,-1,1,1,-1,1,1,1,1,1$ \\
 & & \\
ABABCDCD:aaaa & 4 & $-1,-1,1,1,0,1,0,0,1,0$ \\
ABABCDCEDE:aaaba & 4 & $-1,-1,1,1,0,1,0,0,1,0$ \\
 & & \\
ABABCDCD:bbbb & 4 & $-1,-1,1,1,0,1,2,2,1,0$ \\
ABABCDCEDE:bbbab & 4 & $-1,-1,1,1,0,1,2,2,1,0$ \\
 & & \\
ABACBDCD:bbbb & 4 & $-1,0,0,1,0,0,2,1,0,0$ \\
ABACBDCEDE:babab & 4 & $-1,0,0,1,0,0,2,1,0,0$ \\
 & & \\
ABACBDCD:aaaa & 4 & $-1,0,0,1,1,1,0,-1,1,1$ \\
ABACBDCEDE:ababa & 4 & $-1,0,0,1,1,1,0,-1,1,1$ \\
 & & \\
ABACBDCD:abba & 4 & $-2,-1,1,2,0,1,3,0,1,0$ \\
ABCADBECDE:bbbbb & 4 & $-2,-1,1,2,0,1,3,0,1,0$ \\
 & & \\
ABACBDCD:baab & 4 & $-2,-1,1,2,1,2,1,2,2,1$ \\
ABCADBECDE:aaaaa & 4 & $-2,-1,1,2,1,2,1,2,2,1$ \\
 & & \\
ABACDBDECE:abbab & 5 & $-1,-1,0,1,1,-1,0,1,1,0,1,1,1,1,1$ \\
ABACDBECED:ababb & 5 & $-1,-1,0,1,1,-1,0,1,1,0,1,1,1,1,1$ \\
 & & \\
ABACDBDECE:abaab & 5 & $-1,-1,0,1,1,-1,1,1,1,1,1,1,0,0,1$ \\
ABACDBECED:babaa & 5 & $-1,-1,0,1,1,-1,1,1,1,1,1,1,0,0,1$
\end{tabular}
\end{center}
\caption{Groups of nanowords with $5$ letters that cannot be
 distinguished by primitive based matrices or coverings. It is an open
 question as to whether they are homotopic to each other within each
 group.}
\label{tab:5groups}
\end{table}
\par
In our enumeration of virtual strings with $5$ crossings we discovered
$8$ pairs and a triplet of nanowords which cannot be
distinguished by the $u$-polynomial, primitive based matrices or
coverings. We list these nanowords in Table~\ref{tab:5groups}. All
the nanowords in the list have trivial $u$-polynomial and trivial
$r$-covering for any $r$ other than $1$. Of course, it is possible that
each pair or triplet of nanowords are actually homotopic to each other, in
which case it would be no surprise that their invariants are
identical. However, we have not yet found a sequence of homotopy moves 
relating any pair of nanowords in the list. It is also possible that there
is another invariant we can use to distinguish the nanowords. 
For example, we have not considered the invariants derived from Weyl
algebras defined by Fenn and Turaev in \cite{Fenn/Turaev:Weyl_algebra}.
\section{Kadokami's statement}\label{sec:kadokami}
We defined a \surfdiag{} as an oriented circle immersed in a surface
with self-intersections limited to transverse
double points. We can generalize this definition by
allowing multiple oriented circles to be immersed in the plane.
For a natural number $n$
we define an $n$-component \surfdiag{} to be a pair $(S,D)$ where $S$ is a
surface and $D$ is the immersion of $n$
oriented circles immersed in $S$ with self-intersections limited to
transverse double points. 
In this section we just consider \surfdiag{}s with a finite number of
double points. 
As before we can define an equivalence relation on such diagrams by
stable homeomorphisms and homotopy moves. We define an
$n$-component virtual string to be an equivalence class of such diagrams
under this equivalence relation. Note that, since applying stable
homeomorphisms or homotopy moves to an $n$-component \surfdiag{} 
does not change the number of components, the number of components is
invariant under this equivalence relation.
\par
We note that, just as for the single component case, we can also define
$n$-component virtual strings via planar diagrams of $n$ curves where
virtual crossings are permitted.
There is also a representation for $n$-component virtual string diagrams
which corresponds to the nanoword representation. Details of this
representation can be found in \cite{Turaev:KnotsAndWords}.
\par
From now on, in this section we will simply write \emph{diagram} to mean
a \surfdiag{} with one or more components.
\par
Given a diagram $(S,D)$, we can cut out the regular
neighbourhood $N(D)$ of $D$ in $S$. As $S$ is oriented, $N(D)$ is an
oriented surface with one or more boundaries. To each boundary of $N(D)$
we attach a disk. The result is an oriented surface $S^\prime$ in which
$D$ is embedded. Then $(S^\prime,D)$ is a diagram which
is stably homeomorphic to $(S,D)$. The surface $S^\prime$ is called the
\emph{canonical surface} for $D$. This surface has the minimum genus
over all surfaces on which $D$ can be drawn without needing virtual
crossings. The construction of this surface is well known. For example,
it is mentioned in \cite{Kadokami:Non-triviality} and
\cite{Turaev:2004}. 
\par
We note that if there is a $1$-move, $3$-move or crossing reducing
$2$-move that can be made on a diagram $(S,D)$, the
move can also be made on the diagram $(S^\prime,D)$ where $S^\prime$
is the canonical surface for $D$. This is because the moves can be drawn
on the plane without requiring virtual crossings. In some cases a
crossing introducing $2$-move necessitates adding a handle to the
surface. Kadokami explains this further in
\cite{Kadokami:Non-triviality}. 
\par
We say that two diagrams are in the same \emph{$3$-class}
if they are related by a finite sequence of $3$-moves
and stable homeomorphisms.
\par
We say that a diagram is \emph{reducible} if it is stably homeomorphic
to a diagram to which we can apply a crossing reducing
$1$-move or $2$-move. We say that a diagram is \emph{reduced} if it is not
reducible and it is not in the same $3$-class as a reducible diagram.
\par
In \cite{Kadokami:Non-triviality}, Kadokami makes the following
statement (his Theorem 3.8) which we paraphrase to use terminology from
this paper.
\begin{stmt}\label{stmt:irredMulti}
For any $n$, two reduced homotopic $n$-component diagrams
 are in the same $3$-class. 
\end{stmt}
In the case where $n$ is $1$ we can interpret this as follows. If
 $\alpha$ and $\beta$ are nanowords representing the same virtual string and
 both nanowords are in irreducible $3$-classes, $\alpha$ and $\beta$ are in
 the same irreducible $3$-class. In other words,
\begin{stmt}\label{stmt:irred}
A virtual string is represented by a unique irreducible $3$-class. 
\end{stmt}
\par
If true, Statement~\ref{stmt:irred} would imply that all nanowords
generated in the computer 
enumeration of virtual string candidates are in fact distinct. This is
because the computer enumeration searches for irreducible $3$-classes
and only outputs the minimal nanoword in each such class. The statement
implies that we can complete the enumeration without having to calculate
invariants. The statement also implies that to discover whether two
nanowords $\alpha$ and $\beta$ are homotopic, it is enough to find the
corresponding irreducible $3$-class in each case and compare
them. Nanowords $\alpha$ and $\beta$ are homotopic if and only if their 
irreducible $3$-classes are the same.
\par
\begin{figure}[hbt]
\begin{center}
\includegraphics{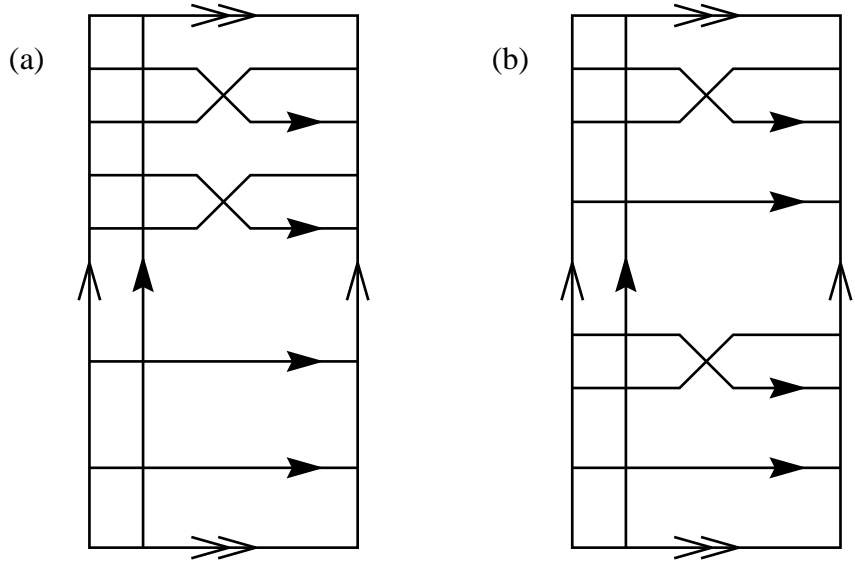}
\caption{A pair of five component diagrams each drawn on
 a torus. In each case the torus is formed by identifying opposite edges
 of the rectangle according to the arrows.}
\label{fig:counterexample}
\end{center}
\end{figure}
\begin{figure}[hbt]
\begin{center}
\includegraphics{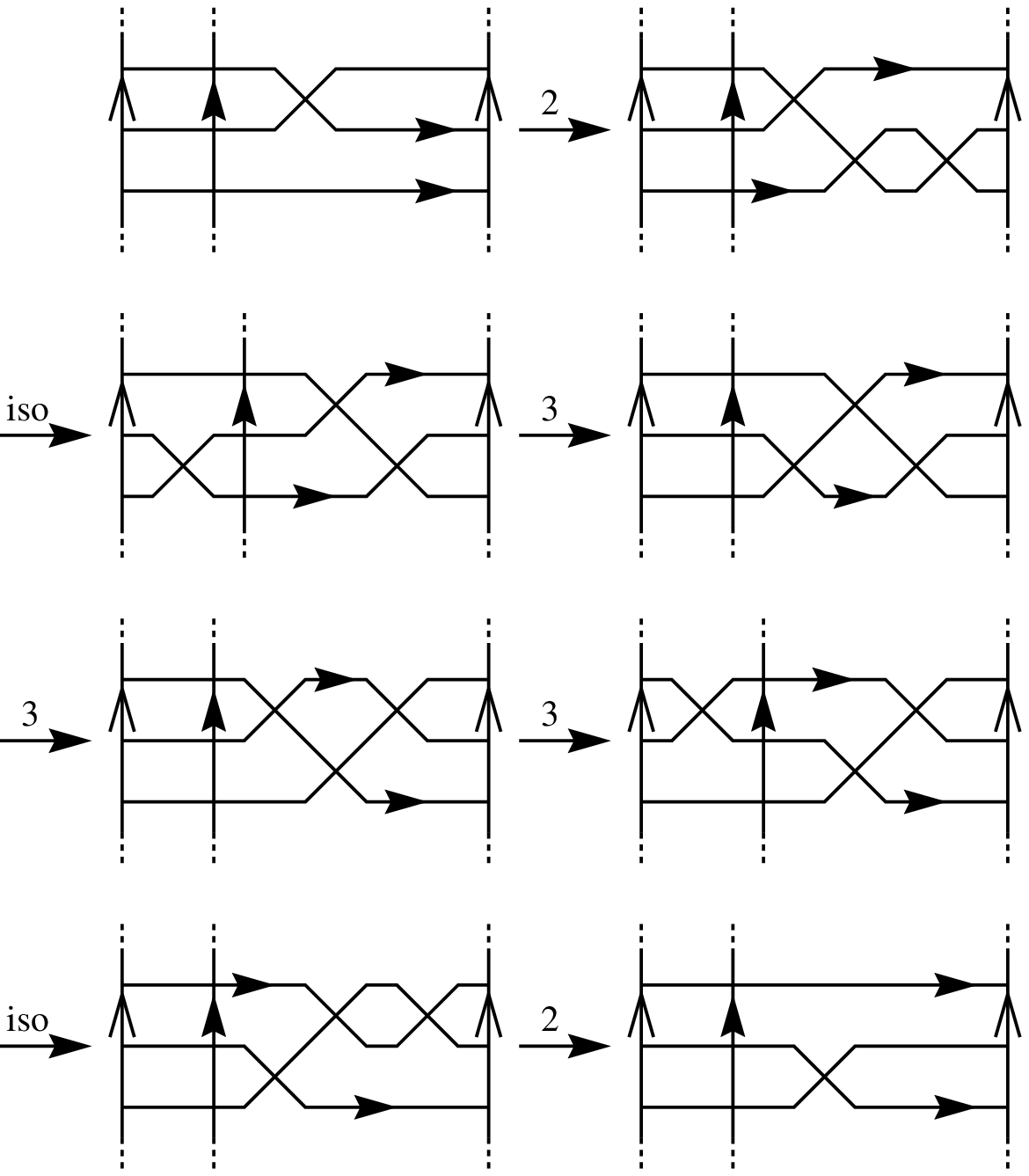}
\caption{A sequence of homotopy moves showing how to swap the positions
 of two components. In each diagram the left and right edges are
 identified to make a cylinder. Labels on the arrows between the
 diagrams indicate which kind of move was used. Here `iso' means
 isotopy achieved by pushing a crossing across the identified edges.} 
\label{fig:homotopicproof}
\end{center}
\end{figure}
Unfortunately, there is a problem. We have found a counter-example for
the multiple component case. We have the following proposition.
\begin{prop}
The two diagrams in Figure~\ref{fig:counterexample} form a counter-example to
 Statement~\ref{stmt:irredMulti}.
\end{prop}
\begin{proof}
To show that this is indeed a counter-example we must first check that
 the two diagrams are not isomorphic.
That is, we need to establish that the diagrams are actually
distinct. Secondly, we should check that the diagrams are homotopic. Thirdly,
we need to show that the two diagrams are reduced. 
Together these show that the two diagrams satisfy the conditions
 of Statement~\ref{stmt:irredMulti}. If the statement is true then the
 diagrams should be related by a sequence of $3$-moves. We will show
 that this is not the case by showing that the two diagrams are in
 different $3$-classes. 
\par
To check that there is not an isomorphism between the two diagrams we
consider the vertical component in each diagram. It is the only
component that intersects with all the other components in each
diagram. We now note that in each diagram there are exactly two
components with which the vertical component intersects only
once. We consider the two points of intersection between these three
components. Notice that in Figure~\ref{fig:counterexample}(a) these two
points are joined by an edge. 
This property is preserved under isomorphism. However, in
 Figure~\ref{fig:counterexample}(b) the two points are not joined by an
 edge. 
So there cannot be an isomorphism between the two diagrams and they are
indeed distinct.
\par
The sequence of diagrams in Figure~\ref{fig:homotopicproof} show how
these two diagrams are related by a sequence of homotopy moves. Thus the
two diagrams are homotopic. Note that in this sequence we have to use a
$2$-move to temporarily increase the number of crossings and later
another $2$-move to reduce them again.
\par
Both diagrams in Figure~\ref{fig:counterexample} are drawn on their
 canonical surfaces. In each diagram the 
curves divide the surface up into regions. For each region we can count
the number of edges. If there was a $1$-move available we would be able
to find a monogon. If there was a $2$-move available we would be able to
find a bigon. It is easy to check that there are no monogons or bigons
in either diagram. Thus both diagrams are irreducible. To verify that
they are both reduced we must now consider their
$3$-classes. 
\par 
\begin{figure}[hbt]
\begin{center}
\includegraphics{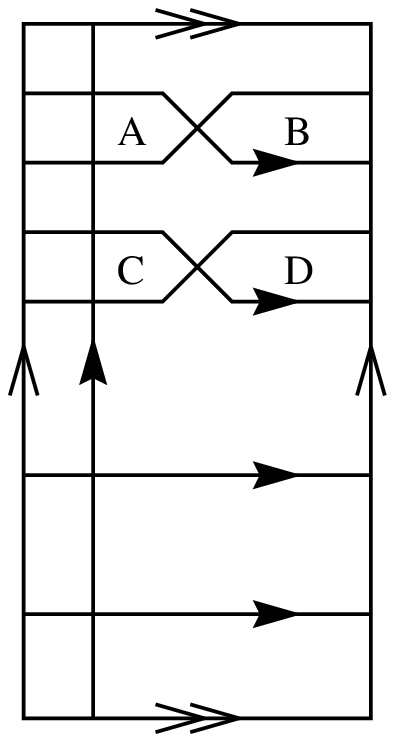}
\caption{The diagram from Figure~\ref{fig:counterexample}(a) with
 triangular regions labelled.}  
\label{fig:explain3class}
\end{center}
\end{figure}
To find possible $3$-moves on a diagram we can look for triangular
regions on the canonical surface. There must be a triangular region on
 the surface for each possible $3$-move. 
Figure~\ref{fig:explain3class} shows
the diagram in Figure~\ref{fig:counterexample}(a) with all the
 triangular regions labelled. 
There are four such regions which means there are four $3$-moves
 available. 
However, we note that if we make the $3$-move associated with any of the
 regions we get a diagram which is isotopic to the diagram we started
 with. Thus the $3$-class of the diagram in
 Figure~\ref{fig:counterexample}(a) just contains a single diagram.  
It is easy to check that the same thing is true for the diagram in
 Figure~\ref{fig:counterexample}(b).
\par
Thus both diagrams are reduced. As the diagrams are
not isomorphic, it is clear from the preceding paragraph that the two
diagrams are in different $3$-classes. Thus the two diagrams are indeed
a counter-example to Statement~\ref{stmt:irredMulti}.
\end{proof}
As yet, we have not found a counter-example for the single component
case. However, we note that if any pair of nanowords in
Table~\ref{tab:5groups} were shown to be homotopic, the pair would be a 
counter-example.
\par
As the proof in Kadokami's paper considers only the
general case of $n$ components and we have given a counter-example to
Statement~\ref{stmt:irredMulti}, it
is clear that there is a problem with the proof. 
It is still possible that Statement~\ref{stmt:irred} is true, but we can
not use it until it has been proved. We hope to consider this problem
further in the future.
\bibliography{mrabbrev,vietnam}
\bibliographystyle{hamsplain}
\end{document}